\documentclass[11pt]{amsart}
\usepackage[a4paper,margin=1in]{geometry}
\usepackage{amsmath,amssymb,amsthm,mathtools,bm}
\usepackage{booktabs}
\usepackage{enumitem}
\usepackage{microtype}
\usepackage[colorlinks=true,linkcolor=blue,citecolor=blue,urlcolor=blue]{hyperref}
\usepackage[nameinlink,noabbrev]{cleveref}
\hypersetup{pdftitle={Dual Geometry of Spherical Designs: Polarity, Self-Polar Rigidity, and Quadrature Structure},pdfauthor={Congpei An}}

\newtheorem{theorem}{Theorem}[section]
\newtheorem{proposition}[theorem]{Proposition}
\newtheorem{lemma}[theorem]{Lemma}
\newtheorem{corollary}[theorem]{Corollary}
\theoremstyle{definition}
\newtheorem{definition}[theorem]{Definition}
\newtheorem{problem}[theorem]{Problem}
\theoremstyle{remark}
\newtheorem{remark}[theorem]{Remark}

\newcommand{\R}{\mathbb R}
\newcommand{\Sph}{\mathbb S}
\newcommand{\Id}{\mathrm I}
\newcommand{\conv}{\operatorname{conv}}
\newcommand{\Verts}{\operatorname{vert}}
\newcommand{\rank}{\operatorname{rank}}
\newcommand{\diag}{\operatorname{diag}}

\newcommand{\dd}{\,\mathrm d}
\newcommand{\norm}[1]{\left\lVert #1\right\rVert}

\newcommand{\ip}[2]{\left\langle #1,#2\right\rangle}
\newcommand{\dH}{d_{\mathrm H}}

\newcommand{\Area}{\operatorname{Area}}

\newcommand{\HS}{\mathrm{HS}}

\title[Dual Geometry of Spherical Designs]{Dual Geometry of Spherical Designs:\\
Polarity, Self-Polar Rigidity, and Quadrature Structure}
\author{Congpei An}
\date{August 2026}

\begin{document}

\begin{abstract}
Let $X=\{x_1,\ldots,x_N\}\subset\Sph^{d-1}$ be a spherical $t$-design, $t\ge2$, and let $P_X\subset\R^d$ be its Minkowski polytope, normalized so that the origin lies in its interior.  Thus the $x_i$ are the outer facet normals of $P_X$ and all facets have equal $(d-1)$-area.  We study the geometry obtained after polarizing this realization.  If $h_i=h_{P_X}(x_i)$, then
\[
 P_X^\circ=\conv\{x_i/h_i:1\le i\le N\},
\]
so the original design is exactly the radial projection of the polar vertex set.  Homogeneous moments therefore admit an exact degree-dependent transfer formula, while Hausdorff sphericity of $P_X$ gives quantitative control of the unweighted moments of $P_X^\circ$.  In dimension three, the known bound $d_{\rm H}(P_X,B)=O(t^{-1/2})$ yields asymptotic isotropy of polar vertex moments throughout the growing range $m=o(\sqrt t)$.

The main part of the paper concerns self-polarity.  If $P_X^\circ=cUP_X$ with $U\in O(d)$, we derive a structured nonnegative slack factorization
\[
 A=c\,hh^T-X^TU^TX,
 \qquad \rank A=d+1,
\]
whose zero pattern is exactly the facet--vertex incidence relation.  For node-transitive designs all support numbers are equal, and the self-polar incidence level is the inradius-to-circumradius ratio.  Combining this observation with the one-dimensional moment problem behind the Fazekas--Levenshtein covering bound gives a quadrature-combinatorics rigidity theorem: the incidence level is bounded below by an explicit Gaussian or Gauss--Radau node, and equality forces every row of twisted inner products to have precisely the corresponding quadrature nodes and weights.  Thus the quadrature weights become incidence frequencies: at equality, the integers $N\lambda_k$ are the corresponding row multiplicities.  We also prove a quantitative near-equality theorem.  If $\delta=r/R-\eta_{t,d}\ge0$ is small, then every twisted inner-product row is $O_{d,t}(\sqrt\delta)$-close in $W_1$ to the extremal Gaussian or Gauss--Radau quadrature measure, and its entries cluster quantitatively around the quadrature nodes.  In particular the proportion of vertices with normalized slack $O_{d,t}(\sqrt\delta)$ from a facet differs from the top quadrature weight by at most $O_{d,t}(\sqrt\delta)$.  This yields an arithmetic stability gap in terms of the distances of $N\lambda_k$ from the integers.  We obtain sharp low-degree consequences, prove that a node-transitive self-polar Minkowski design polytope in dimension three is necessarily a tetrahedron, and show that among regular polytopal designs in dimensions at least three the only self-polar examples are the regular simplex and the exceptional $24$-cell.  The latter realizes the degree-five Gaussian equality pattern $(1/4,1/2,1/4)$.
\end{abstract}

\maketitle

\section{Introduction}

A spherical $t$-design is a finite subset of the sphere whose equal-weight empirical measure agrees with normalized spherical measure on every polynomial of degree at most $t$.  The definition is analytic, but it has a direct convex-geometric realization.  For $t\ge2$, the first two design moments give the balance and nondegeneracy hypotheses of Minkowski's existence theorem.  Consequently, a spherical design $X\subset\Sph^{d-1}$ determines, up to translation, a convex polytope $P_X$ whose outer facet normals are the nodes of $X$ and whose facets all have the same area.  In dimension three this realization was developed quantitatively in \cite{An2026}: the surface-area measure of $P_X$ agrees with that of the ball through degree $t$, and after Steiner normalization one has
\[
 \dH(P_X,B)=O(t^{-1/2}).
\]
The construction also extends to higher dimensions, with the corresponding inverse-Minkowski stability scale recorded in \cite{An2026}.

The purpose of the present paper is to investigate a different geometric operation: \emph{polar duality}.  The basic formula is elementary but structurally decisive.  If
\[
 P_X=\bigcap_{i=1}^N\{y:\ip{x_i}{y}\le h_i\},
 \qquad h_i=h_{P_X}(x_i)>0,
\]
then
\[
 P_X^\circ=\conv\left\{\frac{x_1}{h_1},\ldots,\frac{x_N}{h_N}\right\}.
\]
Thus a single spherical design simultaneously encodes the facet-normal geometry of $P_X$ and the radial vertex geometry of its polar.  The first question is quantitative: how much of the spherical moment structure survives when the polar vertices are used without radial normalization?  The second question is rigid: what extra constraints arise if the Minkowski polytope is self-polar, namely
\[
 P_X^\circ=cUP_X,
 \qquad c>0,\quad U\in O(d)?
\]

Our first group of results answers the moment-transfer question.  For every homogeneous polynomial $H_m$ of degree $m\le t$, the polar vertices $v_i=x_i/h_i$ satisfy
\[
 \frac1N\sum_{i=1}^Nh_i^mH_m(v_i)
 =\int_{\Sph^{d-1}}H_m(u)\dd\sigma(u).
\]
This is an exact degree-dependent weighted identity.  More significantly, if $\dH(P_X,B)\le\delta<1$, the unweighted polar moment tensor satisfies
\[
 \left\|\frac1N\sum_i v_i^{\otimes m}
      -\int u^{\otimes m}\dd\sigma(u)\right\|_{\HS}
 \le (1-\delta)^{-m}-1.
\]
In the three-dimensional setting of \cite{An2026}, $\delta=O(t^{-1/2})$; hence all polar moments in the growing range $m=o(\sqrt t)$ become asymptotically isotropic.  This gives a direct transfer from inverse Minkowski stability to a quantitative statement about the vertices of the polar body.

The self-polar case exposes a sharper structure.  After matching polar vertices with the vertices of $cUP_X$, the vertex corresponding to $x_j$ is
\[
 y_j=\frac{1}{c h_j}U^Tx_j.
\]
Hence the radial vertex directions of $P_X$ form the rotated spherical $t$-design $U^TX$.  More importantly, the matrix
\[
 A_{ij}=c h_i h_j-\ip{Ux_i}{x_j}
\]
is nonnegative, has rank $d+1$, and has zeros precisely at facet--vertex incidences.  In matrix form,
\[
 A=c\,hh^T-X^TU^TX,
\]
where $X$ is the $d\times N$ matrix with columns $x_i$.  This identifies a spherical-design Gram-type matrix inside the slack geometry of a self-polar polytope.  Degree two exactness further implies that the twisted Gram term has exactly $d$ nonzero singular values, all equal to $N/d$.

For node-transitive designs, the geometry simplifies dramatically.  Minkowski uniqueness forces all support numbers to equal a common value $h$.  Self-polarity then makes every vertex lie on a common sphere of radius $R=(ch)^{-1}$, while every facet is tangent to the sphere of radius $r=h$.  The common incidence level is therefore
\[
 \alpha:=ch^2=\frac rR,
\]
and for every $i$,
\[
 \alpha=\max_j\ip{Ux_i}{x_j}.
\]
At this point a classical covering-radius bound for spherical designs becomes a new constraint on self-polar geometry.  We give a short self-contained proof of the relevant one-dimensional moment theorem, equivalent to the Fazekas--Levenshtein bound \cite{FazekasLevenshtein1995}.  If $t=2n+1$, the lower threshold $\eta_{t,d}$ is the largest Gaussian node for the spherical marginal measure; if $t=2n$, it is the largest Gauss--Radau node with fixed endpoint $-1$.  We obtain
\[
 \boxed{\frac rR=\alpha\ge\eta_{t,d}.}
\]
Equality is rigid: for every facet row, the numbers $\ip{Ux_i}{x_j}$ are exactly the corresponding Gaussian or Gauss--Radau nodes, and their multiplicities are $N$ times the quadrature weights.  Thus a continuous quadrature rule controls discrete face cardinalities.  The covering threshold itself is classical; the new point here is its appearance as a self-polar facet--vertex incidence threshold and its coupling to the Minkowski realization.

We then establish a quantitative stability theorem around this equality regime.  Put $\delta=r/R-\eta_{t,d}$.  For fixed $d$ and $t$, if $\delta$ is small, every empirical row measure
\[
 \nu_i=\frac1N\sum_{j=1}^N\delta_{\ip{Ux_i}{x_j}}
\]
is $O_{d,t}(\sqrt\delta)$-close in $W_1$ to the extremal Gaussian or Gauss--Radau rule.  More precisely, the row entries have mean-square distance $O_{d,t}(\delta)$ from the quadrature node set, their cluster multiplicities differ from $N\lambda_k$ by $O_{d,t}(N\sqrt\delta)$, and the number of vertices with normalized facet slack at most $O_{d,t}(\sqrt\delta)$ differs from $N\lambda_+$ by the same order.  Consequently
\[
 \delta\ge \frac{c_{d,t}}{N^2}
 \max_k \operatorname{dist}(N\lambda_k,\mathbb Z)^2.
\]
Thus the exact integrality constraint at equality persists as a quantitative arithmetic obstruction to near equality.

Two examples are especially revealing.  For $t=2$, $\eta_{2,d}=1/d$, and equality forces $N=d+1$ and $X$ to be a regular simplex.  For $t=5$ and $d=4$, $\eta_{5,4}=1/\sqrt2$ and the Gaussian weights are $(1/4,1/2,1/4)$.  The regular $24$-cell realizes this pattern: every relevant facet row contains $6$, $12$, and $6$ inner products at $1/\sqrt2$, $0$, and $-1/\sqrt2$, respectively.  The equality therefore encodes the six vertices of each octahedral facet.

We also derive two combinatorial rigidity consequences.  If $d\ge3$ and a self-polar Minkowski design polytope is simple, then it is necessarily a simplex and the design has strength at most two.  In dimension three, node-transitivity alone is enough: a node-transitive self-polar Minkowski design polytope has equally many vertices and facets and is vertex-transitive; Euler's formula then forces four vertices, hence a tetrahedron.  Thus no node-transitive spherical $t$-design with $t\ge3$ can produce a self-polar Minkowski polytope in $\R^3$.

\paragraph{Relation to existing work.}
Self-polar polytopes, their realizations, and their facial structure have been studied systematically; see Jensen \cite{Jensen2021}.  The relation between self-duality and structured (in particular symmetric or positive-semidefinite, after suitable scaling) slack matrices has also been developed for polyhedral cones by Gouveia--Louren\c{c}o \cite{GouveiaLourenco2023}.  Our slack theorem below uses a different additional input: the Minkowski realization of a spherical design forces the explicit rank-one-minus-twisted-tight-frame decomposition $A=c hh^T-X^TU^TX$.  The covering bound used below goes back to Fazekas--Levenshtein \cite{FazekasLevenshtein1995}, and its connection with spherical polarization is developed in \cite{BoyvalenkovEtAl2023}.  For odd strength, Borodachov \cite{Borodachov2024} characterizes spherical designs attaining the Fazekas--Levenshtein covering bound and relates equality to Gauss--Gegenbauer quadrature, including arithmetic restrictions on the quadrature weights; that paper also reviews the corresponding strongly sharp even-strength equality regime.  We therefore do not claim novelty for the one-dimensional equality mechanism itself.  Recent work of Ehler \cite{Ehler2025} uses vertices and edges of dual pairs of regular polytopes to construct hybrid spherical designs, while Waldron \cite{Waldron2025} studies Gramian formulations of weighted spherical designs.  Our viewpoint is different: we first reconstruct a convex polytope from the \emph{surface-area measure} of a spherical design, then polarize that specific Minkowski realization, and finally use self-polarity to couple design moments, slack matrices, quadrature nodes, and facet--vertex incidences.  Accordingly, the homogeneous transfer identity below is best viewed as a geometric specialization.  The principal new structure lies in the self-polar slack factorization, the conversion of quadrature data into facet--vertex incidence data, and the quantitative near-equality theorem, which turns a small covering-threshold defect into stability of the entire row distribution and a near-incidence statement.

The paper is organized as follows.  \Cref{sec:prelim} fixes the spherical-design and polarity framework.  \Cref{sec:moment} proves exact and quantitative polar moment transfer.  \Cref{sec:selfpolar} develops self-polar design polytopes and the structured slack matrix.  \Cref{sec:transitive} treats node-transitive designs and gives the inradius/circumradius interpretation.  \Cref{sec:quadrature} proves the moment threshold and the quadrature-combinatorics rigidity theorem.  \Cref{sec:stability} develops quantitative near-equality stability and the arithmetic gap theorem.  \Cref{sec:rigidity} gives low-degree and three-dimensional rigidity.  \Cref{sec:regular} treats regular examples, including the $24$-cell.  We conclude with open problems suggested by the new duality framework.

\section{Preliminaries: designs, Minkowski polytopes, and polarity}
\label{sec:prelim}

Let $\sigma$ be normalized surface measure on $\Sph^{d-1}$ and let
\[
 \omega_{d-1}=\Area(\Sph^{d-1}).
\]
A finite set $X=\{x_1,\ldots,x_N\}\subset\Sph^{d-1}$ is a spherical $t$-design if
\begin{equation}\label{eq:design}
 \frac1N\sum_{i=1}^Np(x_i)
 =\int_{\Sph^{d-1}}p(u)\dd\sigma(u)
\end{equation}
for every polynomial of total degree at most $t$ restricted to the sphere.  See Delsarte--Goethals--Seidel \cite{DGS1977}.

The first two moments, valid for every $t\ge2$, are
\begin{equation}\label{eq:firstsecond}
 \frac1N\sum_{i=1}^Nx_i=0,
 \qquad
 \frac1N\sum_{i=1}^Nx_ix_i^T=\frac1d\Id_d.
\end{equation}
Let $\mu_X=N^{-1}\sum_i\delta_{x_i}$.

\begin{theorem}[Minkowski realization]\label{thm:minkowski}
Let $X\subset\Sph^{d-1}$ be a spherical $t$-design, $t\ge2$.  There exists a full-dimensional convex polytope $P_X\subset\R^d$, unique up to translation, such that
\begin{equation}\label{eq:area-measure}
 S_{P_X}=\omega_{d-1}\mu_X
 =\frac{\omega_{d-1}}{N}\sum_{i=1}^N\delta_{x_i}.
\end{equation}
Thus $P_X$ has exactly $N$ facets, the $i$th outer unit facet normal is $x_i$, and every facet has $(d-1)$-area $\omega_{d-1}/N$.
\end{theorem}

\begin{proof}
The first identity in \eqref{eq:firstsecond} is the balance condition in Minkowski's theorem.  The second implies that the support of $\mu_X$ is not contained in any great subsphere: if $a\ne0$ and $\ip{a}{x_i}=0$ for all $i$, then
\[
 0=\frac1N\sum_i\ip{a}{x_i}^2=\frac1d|a|^2,
\]
a contradiction.  Minkowski's existence theorem therefore gives a full-dimensional convex body with surface area measure \eqref{eq:area-measure}, unique up to translation.  Finite support makes the body a polytope, and each positive atom corresponds to one facet.  See Schneider \cite{Schneider2014}.  \qedhere
\end{proof}

We use the Steiner-normalized representative $s(P_X)=0$ whenever a canonical translation is needed.  Since the Steiner point of a full-dimensional convex body belongs to its interior, this normalization guarantees
\[
 0\in\operatorname{int}P_X.
\]
All polar bodies below are taken with respect to this origin unless otherwise stated.

For a convex body $P$ containing the origin in its interior, its polar is
\[
 P^\circ=\{z\in\R^d:\ip{z}{y}\le1\text{ for every }y\in P\}.
\]
If $P$ has the irredundant half-space representation
\begin{equation}\label{eq:Hrep}
 P=\bigcap_{i=1}^N\{y:\ip{x_i}{y}\le h_i\},
 \qquad h_i>0,
\end{equation}
then standard polarity gives
\begin{equation}\label{eq:polarvertices}
 P^\circ=\conv\left\{\frac{x_i}{h_i}:1\le i\le N\right\}.
\end{equation}
Each $x_i/h_i$ is a vertex because \eqref{eq:Hrep} is irredundant.

For $P=P_X$, we write
\[
 h_i:=h_P(x_i),
 \qquad
 v_i:=\frac{x_i}{h_i}\in\Verts(P^\circ).
\]
Thus
\begin{equation}\label{eq:radialprojection}
 \frac{v_i}{|v_i|}=x_i.
\end{equation}
In particular, the spherical design is exactly the radial projection of the vertices of the polar Minkowski polytope.

\section{Polar vertices and moment transfer}
\label{sec:moment}

\subsection{Exact homogeneous transfer}

\begin{theorem}[Polar homogeneous moment transfer]\label{thm:graded}
Let $X\subset\Sph^{d-1}$ be a spherical $t$-design and $P=P_X$ as above.  For every integer $0\le m\le t$ and every homogeneous polynomial $H_m$ of degree $m$,
\begin{equation}\label{eq:graded}
 \boxed{
 \frac1N\sum_{i=1}^N h_i^mH_m(v_i)
 =\int_{\Sph^{d-1}}H_m(u)\dd\sigma(u).}
\end{equation}
Equivalently, in the $m$th symmetric tensor power,
\begin{equation}\label{eq:tensorgraded}
 \boxed{
 \frac1N\sum_{i=1}^N h_i^m v_i^{\otimes m}
 =M_{d,m}:=\int_{\Sph^{d-1}}u^{\otimes m}\dd\sigma(u).}
\end{equation}
\end{theorem}

\begin{proof}
Since $v_i=x_i/h_i$ and $H_m$ is homogeneous,
\[
 h_i^mH_m(v_i)=H_m(x_i).
\]
Summing and applying the design identity \eqref{eq:design} proves \eqref{eq:graded}.  Testing \eqref{eq:graded} against arbitrary symmetric $m$-linear forms gives \eqref{eq:tensorgraded}.  \qedhere
\end{proof}

\begin{remark}
The weight $h_i^m$ depends on the degree.  Thus \cref{thm:graded} should not be confused with a single weighted cubature rule exact in all degrees simultaneously.  It is a graded identity created by the radial rescaling inherent in polarity.
\end{remark}

\subsection{Quantitative unweighted moments}

Define the unweighted polar vertex moment tensor
\[
 \mathcal V_m(P^\circ):=\frac1N\sum_{i=1}^Nv_i^{\otimes m}.
\]
We use the Hilbert--Schmidt norm on tensor powers, for which
$\|x^{\otimes m}\|_{\HS}=|x|^m$.

\begin{theorem}[Quantitative polar moment transfer]\label{thm:momenttransfer}
Let $X$ be a spherical $t$-design and suppose
\[
 \delta:=\dH(P_X,B)<1.
\]
Then for every $m\le t$,
\begin{equation}\label{eq:momentbound}
 \boxed{
 \norm{\mathcal V_m(P_X^\circ)-M_{d,m}}_{\HS}
 \le (1-\delta)^{-m}-1
 \le \exp\left(\frac{m\delta}{1-\delta}\right)-1.}
\end{equation}
Moreover,
\begin{equation}\label{eq:polarHausdorff}
 \frac1{1+\delta}B\subset P_X^\circ\subset\frac1{1-\delta}B,
 \qquad
 \dH(P_X^\circ,B)\le\frac{\delta}{1-\delta}.
\end{equation}
\end{theorem}

\begin{proof}
Because the Hausdorff distance of convex bodies equals the uniform distance of their support functions,
\[
 1-\delta\le h_P(u)\le1+\delta
 \qquad(u\in\Sph^{d-1}).
\]
In particular the same holds for each $h_i$.  Since $m\le t$,
\[
 M_{d,m}=\frac1N\sum_i x_i^{\otimes m}.
\]
Therefore
\[
 \mathcal V_m(P^\circ)-M_{d,m}
 =\frac1N\sum_i(h_i^{-m}-1)x_i^{\otimes m}.
\]
Taking Hilbert--Schmidt norms gives
\[
 \norm{\mathcal V_m(P^\circ)-M_{d,m}}_{\HS}
 \le\max_i|h_i^{-m}-1|
 \le(1-\delta)^{-m}-1.
\]
Finally,
\[
 -\log(1-\delta)
 =\int_0^\delta\frac{\dd s}{1-s}
 \le\frac{\delta}{1-\delta},
\]
which proves the exponential estimate.

The inclusions $(1-\delta)B\subset P\subset(1+\delta)B$ reverse under polarity and give the first part of \eqref{eq:polarHausdorff}; the displayed Hausdorff bound follows immediately.  \qedhere
\end{proof}

\begin{corollary}[Growing moment window in dimension three]\label{cor:window}
Let $X_t\subset\Sph^2$ be any sequence of spherical $t$-designs and let $P_t$ be the corresponding Steiner-normalized Minkowski polytopes.  Using the estimate
\[
 \dH(P_t,B)=O(t^{-1/2})
\]
from \cite{An2026}, one has, for each fixed $m$,
\[
 \mathcal V_m(P_t^\circ)=M_{3,m}+O_m(t^{-1/2}).
\]
More generally, if $m=m(t)=o(\sqrt t)$, then
\[
 \norm{\mathcal V_{m(t)}(P_t^\circ)-M_{3,m(t)}}_{\HS}\longrightarrow0.
\]
\end{corollary}

\begin{proof}
Insert $\delta_t=O(t^{-1/2})$ into \eqref{eq:momentbound}.  If $m\delta_t\to0$, the exponential upper bound tends to zero.  \qedhere
\end{proof}

\begin{remark}[Higher dimensions]
The higher-dimensional estimate in \cite{An2026} is stated on $\Sph^q\subset\R^{q+1}$ as $O(t^{-1/q})$.  In the notation of the present paper, where the ambient dimension is $d$, this corresponds to a moment window $m=o(t^{1/(d-1)})$ whenever that inverse-Minkowski rate is used.
\end{remark}

\section{Self-polar design polytopes and a structured slack matrix}
\label{sec:selfpolar}

\begin{definition}[Self-polarity up to scale]\label{def:selfpolar}
A convex body $P\subset\R^d$ containing the origin in its interior is \emph{orthogonally self-polar up to scale} if
\begin{equation}\label{eq:selfpolar}
 P^\circ=cUP
\end{equation}
for some $c>0$ and $U\in O(d)$.  When the scale of $P$ is free, it can be normalized away; here we retain $c$ because the Minkowski surface-area normalization fixes the scale.
\end{definition}

This is the scaled version of the self-polarity studied, for example, in Jensen \cite{Jensen2021}.

\begin{theorem}[Facet--vertex design duality]\label{thm:bidual}
Let $X=\{x_1,\ldots,x_N\}$ be a spherical $t$-design, $t\ge2$, and let $P=P_X$ satisfy \eqref{eq:selfpolar}.  Then $P$ has exactly $N$ vertices.  After relabeling them as $y_1,\ldots,y_N$,
\begin{equation}\label{eq:vertexformula}
 \boxed{y_j=\frac{1}{c h_j}U^Tx_j,}
 \qquad
 \boxed{\frac{y_j}{|y_j|}=U^Tx_j.}
\end{equation}
Consequently, the radial vertex directions of $P$ form a spherical $t$-design orthogonally congruent to the facet-normal design $X$.

Furthermore,
\begin{equation}\label{eq:incidenceineq}
 \ip{Ux_i}{x_j}\le c h_i h_j
 \qquad(1\le i,j\le N),
\end{equation}
and equality holds if and only if the vertex $y_j$ lies on the facet with outer normal $x_i$.
Finally,
\begin{equation}\label{eq:U2sym}
 U^2P=P.
\end{equation}
\end{theorem}

\begin{proof}
By \eqref{eq:polarvertices}, the vertices of $P^\circ$ are exactly $x_j/h_j$.  On the other hand, \eqref{eq:selfpolar} gives
\[
 \Verts(P^\circ)=cU\Verts(P).
\]
Hence $P$ has $N$ vertices, and after relabeling we may impose
\[
 cUy_j=\frac{x_j}{h_j}.
\]
This is exactly \eqref{eq:vertexformula}.  Orthogonal invariance of spherical designs proves the design assertion.

Since $y_j\in P$, the $i$th facet inequality gives
\[
 \ip{x_i}{y_j}\le h_i.
\]
Substituting \eqref{eq:vertexformula} yields
\[
 \frac1{c h_j}\ip{x_i}{U^Tx_j}\le h_i,
\]
which is \eqref{eq:incidenceineq}.  Equality in this inequality is precisely the condition $y_j\in F_i$.

Finally,
\[
 P=(P^\circ)^\circ=(cUP)^\circ=c^{-1}UP^\circ=U^2P,
\]
where orthogonality of $U$ is used in the polar transformation rule.  \qedhere
\end{proof}

The next result packages the preceding inequalities into a single matrix with a prescribed low rank.

Let
\[
 X=[x_1\ \cdots\ x_N]\in\R^{d\times N},
 \qquad
 h=(h_1,\ldots,h_N)^T,
 \qquad
 D=\diag(h_1,\ldots,h_N).
\]

\begin{theorem}[Self-polar design slack factorization]\label{thm:slack}
Under the hypotheses of \cref{thm:bidual}, define
\begin{equation}\label{eq:Adef}
 A_{ij}:=c h_i h_j-\ip{Ux_i}{x_j}.
\end{equation}
Then
\begin{equation}\label{eq:Amatrix}
 \boxed{A=c\,hh^T-X^TU^TX.}
\end{equation}
Moreover:
\begin{enumerate}[label=\textup{(\roman*)}]
\item $A$ is entrywise nonnegative;
\item $A_{ij}=0$ if and only if $y_j\in F_i$;
\item $\rank A=d+1$;
\item if $H=\sum_i h_i$, then
\begin{equation}\label{eq:rowsums}
 A\mathbf1=A^T\mathbf1=cHh;
\end{equation}
\item the matrix $B:=X^TU^TX$ has rank $d$ and its nonzero singular values are all $N/d$.
\end{enumerate}
If $U^2x_j=x_{\pi(j)}$ for the permutation $\pi$ induced by the symmetry $U^2P=P$, and $\Pi$ is its permutation matrix, then
\begin{equation}\label{eq:twistsym}
 A^T=A\Pi.
\end{equation}
In particular, an involutory polarity $U^2=\Id$ gives a symmetric design slack matrix.
\end{theorem}

\begin{proof}
The matrix identity follows from
\[
 (X^TU^TX)_{ij}=x_i^TU^Tx_j=\ip{Ux_i}{x_j}.
\]
Parts (i) and (ii) are exactly \eqref{eq:incidenceineq} and its equality case.

Let $S$ be the usual slack matrix of $P$, with rows indexed by facets and columns by the matched vertices $y_j$:
\[
 S_{ij}=h_i-\ip{x_i}{y_j}.
\]
Using \eqref{eq:vertexformula},
\[
 c h_j S_{ij}
 =c h_i h_j-\ip{Ux_i}{x_j}=A_{ij}.
\]
Hence
\begin{equation}\label{eq:ASD}
 A=S(cD).
\end{equation}
The diagonal matrix $cD$ is invertible.  A slack matrix of a full-dimensional $d$-polytope has rank $d+1$; see, for example, \cite{GouveiaEtAl2013}.  Equivalently, in the present setting this also follows directly from the factorization
\[
 S_{ij}=[h_i,-x_i^T]\binom{1}{y_j}
\]
and the fact that both the facet and vertex augmented configurations span dimension $d+1$.  Thus $\rank A=d+1$.

Since $X\mathbf1=0$ by the degree-one design condition,
\[
 (X^TU^TX)\mathbf1=0,
 \qquad
 (X^TU^TX)^T\mathbf1=0.
\]
This proves \eqref{eq:rowsums}.

For the singular-value statement, degree-two exactness gives
\[
 XX^T=\frac Nd\Id_d.
\]
Set $F=\sqrt{d/N}\,X$, so $FF^T=\Id_d$.  Then
\[
 B=\frac Nd F^TU^TF.
\]
If $T=F^TU^TF$, then
\[
 TT^T=F^TU^TFF^TUF=F^TF,
\]
which is an orthogonal projector of rank $d$.  Hence $T$ has exactly $d$ singular values equal to $1$, proving (v).

Finally $U^2P=P$ implies $U^2X=X\Pi$ after relabeling facet normals, and support-number invariance gives $h_{\pi(j)}=h_j$.  Thus
\[
 A_{i,\pi(j)}
 =c h_i h_j-\ip{Ux_i}{U^2x_j}
 =c h_i h_j-\ip{Ux_j}{x_i}
 =A^T_{ij},
\]
which is \eqref{eq:twistsym}.  \qedhere
\end{proof}

\begin{remark}[A low-rank nonnegative design matrix]
The decomposition \eqref{eq:Amatrix} is rigid in three simultaneous senses: $A$ is nonnegative, its zero pattern is a polytope incidence relation, and after subtraction of the rank-one term $c hh^T$ one obtains a rank-$d$ partial-isometry Gram structure forced by the spherical $2$-design condition.  This coupling does not occur in a general self-polar polytope without the design moment equations.
\end{remark}

\section{Node-transitive designs and the self-polar incidence level}
\label{sec:transitive}

\begin{definition}
A spherical design $X\subset\Sph^{d-1}$ is \emph{node-transitive} if a subgroup $G\le O(d)$ acts transitively on $X$.
\end{definition}

\begin{proposition}[Constant support numbers]\label{prop:constant-h}
Let $X$ be node-transitive and let $P_X$ be Steiner-normalized.  Then
\[
 h_{P_X}(x_1)=\cdots=h_{P_X}(x_N)=:h.
\]
If $C=\conv X$, then
\begin{equation}\label{eq:P=hpolarC}
 \boxed{P_X=hC^\circ,\qquad P_X^\circ=h^{-1}C.}
\end{equation}
Consequently, $P_X$ is self-polar up to scale and orthogonal transformation if and only if $C$ is.
\end{proposition}

\begin{proof}
For $g\in G$, invariance of $X$ gives $g_\#S_{P_X}=S_{P_X}$.  Minkowski uniqueness says that $gP_X$ is a translate of $P_X$.  Steiner normalization removes the translation, so $gP_X=P_X$.  Thus $h_P(gx)=h_P(x)$; transitivity on $X$ gives the common value $h$.

Now
\[
 P=\bigcap_i\{y:\ip{x_i}{y}\le h\}
 =h\left(\conv X\right)^\circ,
\]
which proves \eqref{eq:P=hpolarC}.  The final equivalence follows by polarizing and rescaling.  \qedhere
\end{proof}

Suppose now that $P_X^\circ=cUP_X$.  Then \eqref{eq:vertexformula} and constant $h$ show that every vertex of $P_X$ has the same Euclidean norm
\[
 R=\frac1{ch},
\]
while every facet hyperplane has distance
\[
 r=h
\]
from the origin.  Thus $r$ and $R$ are respectively the centered inradius and circumradius.  Define
\begin{equation}\label{eq:alpha}
 \alpha:=ch^2=\frac rR.
\end{equation}

\begin{proposition}[Incidence level as a twisted polarization value]\label{prop:alpha}
Under node-transitivity and self-polarity,
\begin{equation}\label{eq:maxrow}
 \boxed{\alpha=\max_{1\le j\le N}\ip{Ux_i}{x_j}}
 \qquad\text{for every }i.
\end{equation}
Moreover,
\[
 y_j\in F_i
 \quad\Longleftrightarrow\quad
 \ip{Ux_i}{x_j}=\alpha.
\]
\end{proposition}

\begin{proof}
With $h_i=h_j=h$, \eqref{eq:incidenceineq} becomes
\[
 \ip{Ux_i}{x_j}\le ch^2=\alpha.
\]
Every facet contains at least one vertex, so equality occurs in every row.  \qedhere
\end{proof}

Thus self-polarity converts a geometric ratio $r/R$ into the largest correlation between a rotated design node and the original design.  The next section constrains this number using only the design moments.

\section{Covering thresholds and quadrature-combinatorics rigidity}
\label{sec:quadrature}

\subsection{The spherical marginal moment problem}

Let $\mu_d$ be the distribution of the first coordinate of a uniformly distributed point on $\Sph^{d-1}$.  Thus
\begin{equation}\label{eq:marginal}
 \dd\mu_d(s)=c_d(1-s^2)^{(d-3)/2}\dd s,
 \qquad -1\le s\le1,
\end{equation}
where $c_d$ normalizes the measure.

For a fixed $u\in\Sph^{d-1}$ and a spherical $t$-design $X$, define
\[
 \nu_u:=\frac1N\sum_{j=1}^N\delta_{\ip{u}{x_j}}.
\]
Then
\begin{equation}\label{eq:momentmatch1d}
 \int q\dd\nu_u=\int q\dd\mu_d
 \qquad\text{for every polynomial }q\text{ of degree }\le t.
\end{equation}

We now define the threshold $\eta_{t,d}$.  Let $p_k$ denote the degree-$k$ orthogonal polynomial for $\mu_d$, with arbitrary positive leading coefficient.  Let $p_k^+$ denote the degree-$k$ orthogonal polynomial for the modified positive measure
\[
 (1+s)\dd\mu_d(s).
\]
For $t\ge1$, set
\begin{equation}\label{eq:eta}
 \eta_{t,d}:=
 \begin{cases}
 \text{largest zero of }p_{n+1},&t=2n+1,\\
 \text{largest zero of }p_n^+,&t=2n.
 \end{cases}
\end{equation}
For odd strength these are the largest Gaussian nodes for $\mu_d$; for even strength they are the largest Gauss--Radau nodes with the endpoint $-1$ prescribed.  This is the spherical form of the Fazekas--Levenshtein covering threshold \cite{FazekasLevenshtein1995}.

\begin{theorem}[Moment threshold and equality quadrature]\label{thm:FLselfcontained}
Let $\nu$ be a probability measure supported on $[-1,\alpha]$ whose moments agree with $\mu_d$ through degree $t$.  Then
\begin{equation}\label{eq:alphageeta}
 \boxed{\alpha\ge\eta_{t,d}.}
\end{equation}

If $t=2n+1$ and equality holds, then $\nu$ is the $(n+1)$-node Gaussian quadrature measure for $\mu_d$: its support is exactly the zeros
\[
 \xi_0<\cdots<\xi_n=\eta_{t,d}
\]
of $p_{n+1}$, with the corresponding positive Gaussian weights $\lambda_0,\ldots,\lambda_n$.

If $t=2n$ and equality holds, then $\nu$ is the $(n+1)$-node Gauss--Radau quadrature measure with prescribed node $-1$: its support is
\[
 -1=\xi_0<\xi_1<\cdots<\xi_n=\eta_{t,d},
\]
where $\xi_1,\ldots,\xi_n$ are the zeros of $p_n^+$, with the corresponding positive Radau weights $\lambda_0,\ldots,\lambda_n$.
\end{theorem}

\begin{proof}
Suppose first that $t=2n+1$.  For every polynomial $q$ of degree at most $n$,
\[
 0\le\int(\alpha-s)q(s)^2\dd\nu(s)
 =\int(\alpha-s)q(s)^2\dd\mu_d(s),
\]
where moment matching is valid because the integrand has degree at most $2n+1$.  Hence
\[
 \alpha\ge
 \sup_{0\ne q\in\mathcal P_n}
 \frac{\int s q(s)^2\dd\mu_d(s)}{\int q(s)^2\dd\mu_d(s)}.
\]
The supremum is the largest eigenvalue of the compression of multiplication by $s$ to $\mathcal P_n$, equivalently the largest zero of $p_{n+1}$; see the standard Jacobi-matrix description of Gaussian quadrature, e.g. Szeg\H{o} \cite{Szego1975}.  This proves \eqref{eq:alphageeta}.

If equality holds, choose an extremal polynomial $q_*$ for the Rayleigh quotient.  Then
\[
 \int(\eta_{t,d}-s)q_*(s)^2\dd\nu(s)=0.
\]
Every summand is nonnegative on the support of $\nu$, so each support point is either $\eta_{t,d}$ or a zero of $q_*$.  By the Christoffel--Darboux formula, $q_*$ is proportional to $p_{n+1}(s)/(s-\eta_{t,d})$.  Thus the support is contained in the $n+1$ zeros of $p_{n+1}$.  Moment exactness through degree $2n+1$ uniquely determines the masses on these nodes, giving the Gaussian quadrature weights; their positivity shows that every node occurs.

Now let $t=2n$.  For every $q\in\mathcal P_{n-1}$,
\[
 0\le\int(\alpha-s)(1+s)q(s)^2\dd\nu(s)
 =\int(\alpha-s)(1+s)q(s)^2\dd\mu_d(s).
\]
Thus $\alpha$ dominates the largest Rayleigh quotient of multiplication by $s$ on $\mathcal P_{n-1}$ in $L^2((1+s)\mu_d)$, namely the largest zero of $p_n^+$.  Equality forces every support point to be either $-1$, the largest zero, or a zero of the corresponding extremal polynomial.  The Christoffel--Darboux argument identifies the remaining points with the other zeros of $p_n^+$, and degree-$2n$ exactness gives the unique Gauss--Radau weights.  \qedhere
\end{proof}

\begin{corollary}[Fazekas--Levenshtein covering form]\label{cor:covering}
For every spherical $t$-design $X\subset\Sph^{d-1}$ and every $u\in\Sph^{d-1}$,
\begin{equation}\label{eq:covering}
 \max_j\ip{u}{x_j}\ge\eta_{t,d}.
\end{equation}
Equivalently, the angular covering radius is at most $\arccos\eta_{t,d}$.
\end{corollary}

\begin{proof}
Apply \cref{thm:FLselfcontained} to \eqref{eq:momentmatch1d} with $\alpha=\max_j\ip{u}{x_j}$.  \qedhere
\end{proof}

\subsection{Self-polarity turns quadrature weights into face numbers}

\begin{theorem}[Quadrature-combinatorics rigidity]\label{thm:quadcomb}
Let $X\subset\Sph^{d-1}$ be a node-transitive spherical $t$-design, $t\ge2$, and let its Steiner-normalized Minkowski polytope satisfy
\[
 P_X^\circ=cUP_X.
\]
Let $r$ and $R$ be the centered inradius and circumradius of $P_X$.  Then
\begin{equation}\label{eq:rRbound}
 \boxed{\frac rR=ch^2\ge\eta_{t,d}.}
\end{equation}

If equality holds, then for each facet $F_i$ the multiset
\begin{equation}\label{eq:rowmultiset}
 \{\ip{Ux_i}{x_j}:1\le j\le N\}
\end{equation}
is exactly the Gaussian quadrature multiset when $t$ is odd and the Gauss--Radau multiset when $t$ is even.  In particular, if $\lambda_+$ is the quadrature weight of the largest node $\eta_{t,d}$, then
\begin{equation}\label{eq:facetsize}
 \boxed{\#\Verts(F_i)=N\lambda_+.}
\end{equation}
More generally, every quadrature weight satisfies the arithmetic constraint
\begin{equation}\label{eq:integrality}
 N\lambda_k\in\mathbb Z.
\end{equation}
\end{theorem}

\begin{proof}
By \cref{prop:alpha}, for each $i$,
\[
 \frac rR=\alpha=\max_j\ip{Ux_i}{x_j}.
\]
Since $Ux_i$ is a unit vector, \cref{cor:covering} gives \eqref{eq:rRbound}.

If equality holds, apply the equality statement of \cref{thm:FLselfcontained} to the empirical measure
\[
 \frac1N\sum_j\delta_{\ip{Ux_i}{x_j}}.
\]
Thus a fraction $\lambda_k$ of the $N$ entries equals the $k$th quadrature node.  The maximum node is precisely the equality level in the facet incidence condition of \cref{prop:alpha}, so its multiplicity is the number of vertices on $F_i$.  This proves \eqref{eq:facetsize}; all multiplicities are integers, giving \eqref{eq:integrality}.  \qedhere
\end{proof}

\begin{remark}[Geometric meaning of equality]
When $r/R=\eta_{t,d}$, every direction $Ux_i$ is a deepest-hole direction for the covering problem of $X$.  Thus equality simultaneously identifies the self-polar dual directions, the extremizers of the covering bound, and the top Gaussian/Radau node.
\end{remark}

For reference, the first few thresholds are
\begin{equation}\label{eq:smallthresholds}
 \eta_{2,d}=\frac1d,
 \qquad
 \eta_{3,d}=\frac1{\sqrt d},
 \qquad
 \eta_{5,d}=\sqrt{\frac3{d+2}}.
\end{equation}
For $t=5$ the Gaussian weights at the three nodes
$-\eta_{5,d},0,\eta_{5,d}$ are
\begin{equation}\label{eq:t5weights}
 \lambda_- =\lambda_+=\frac{d+2}{6d},
 \qquad
 \lambda_0=\frac{2(d-1)}{3d}.
\end{equation}
These follow from symmetry and the second spherical moment.

\section{Quantitative stability near the quadrature threshold}
\label{sec:stability}

The equality statement in \cref{thm:quadcomb} is discrete: an extremal row is supported exactly on the Gaussian or Gauss--Radau nodes.  We now show that this phenomenon is quantitatively stable.  The argument is a robust version of the one-dimensional moment proof in \cref{thm:FLselfcontained}.

Let
\[
 \Xi_{t,d}=\{\xi_0<\xi_1<\cdots<\xi_n=\eta_{t,d}\}
\]
denote the extremal quadrature node set from \cref{thm:FLselfcontained}, and let
\[
 \nu_{t,d}^*:=\sum_{k=0}^n\lambda_k\delta_{\xi_k}
\]
be the corresponding Gaussian or Gauss--Radau probability measure.  We write $W_1$ for the Wasserstein distance on $[-1,1]$ with cost $|s-r|$.

The following elementary coercivity lemma is the key point.

\begin{lemma}[Coercivity of the extremal certificate]\label{lem:certificate-coercive}
There exist constants $\varepsilon_0,c_0>0$, depending only on $d$ and $t$, with the following property.  For every
\[
 \eta_{t,d}\le\alpha\le\eta_{t,d}+\varepsilon_0,
\]
there is a node set $\Xi_\alpha$ obtained from $\Xi_{t,d}$ by replacing its largest node $\eta_{t,d}$ with $\alpha$, and a nonnegative polynomial $\Phi_\alpha$ on $[-1,\alpha]$ such that
\begin{equation}\label{eq:certificate-coercive}
 \Phi_\alpha(s)\ge c_0\,\operatorname{dist}(s,\Xi_\alpha)^2,
 \qquad -1\le s\le\alpha.
\end{equation}
Moreover, if $\delta:=\alpha-\eta_{t,d}$ and $\nu$ is any probability measure supported on $[-1,\alpha]$ whose moments agree with $\mu_d$ through degree $t$, then
\begin{equation}\label{eq:certificate-defect}
 \int\Phi_\alpha\,d\nu=\delta.
\end{equation}
\end{lemma}

\begin{proof}
Suppose first that $t=2n+1$.  Let $q_*$ be the degree-$n$ extremizer in the Rayleigh quotient appearing in the proof of \cref{thm:FLselfcontained}, normalized by
\[
 \int q_*^2\,d\mu_d=1,
 \qquad
 \int s q_*^2\,d\mu_d=\eta_{t,d}.
\]
Equivalently, $q_*$ is proportional to $p_{n+1}(s)/(s-\eta_{t,d})$.  Put
\[
 \Phi_\alpha(s):=(\alpha-s)q_*(s)^2.
\]
Its zeros on $[-1,\alpha]$ are the $n$ zeros of $q_*$ together with $\alpha$; at $\alpha=\eta_{t,d}$ these are precisely the Gaussian nodes.  Moment matching through degree $2n+1$ gives
\[
 \int\Phi_\alpha\,d\nu
 =\int\Phi_\alpha\,d\mu_d
 =\alpha-\eta_{t,d}.
\]

For $t=2n$, let $q_*$ be the degree-$(n-1)$ extremizer for multiplication by $s$ in $L^2((1+s)\mu_d)$, normalized by
\[
 \int(1+s)q_*^2\,d\mu_d=1,
 \qquad
 \int s(1+s)q_*^2\,d\mu_d=\eta_{t,d}.
\]
Then $q_*$ is proportional to $p_n^+(s)/(s-\eta_{t,d})$, and we set
\[
 \Phi_\alpha(s):=(\alpha-s)(1+s)q_*(s)^2.
\]
Its zero set consists of $-1$, the zeros of $q_*$, and $\alpha$, which at $\alpha=\eta_{t,d}$ is exactly the Gauss--Radau node set.  Moment matching through degree $2n$ again gives \eqref{eq:certificate-defect}.

It remains to prove \eqref{eq:certificate-coercive}.  Choose $\varepsilon_0>0$ smaller than one quarter of the minimal separation between distinct nodes of $\Xi_{t,d}$ and smaller than $1-\eta_{t,d}$.  The zeros of $q_*$ are simple.  Hence, at every fixed interior node $\zeta$ of $\Xi_{t,d}$, the quotient
\[
 \frac{\Phi_\alpha(s)}{(s-\zeta)^2}
\]
extends continuously to $s=\zeta$, and its limiting value is strictly positive, uniformly for $\alpha\in[\eta_{t,d},\eta_{t,d}+\varepsilon_0]$.  At the moving node $s=\alpha$ the polynomial vanishes linearly, so $\Phi_\alpha(s)/(\alpha-s)^2\to+\infty$ as $s\uparrow\alpha$; in the even case the same statement holds at the fixed endpoint $-1$.  On the complement of small disjoint neighborhoods of these nodes, $\Phi_\alpha$ is strictly positive and depends continuously on $(\alpha,s)$ over a compact set.  Taking the minimum of the resulting positive lower bounds gives a constant $c_0>0$ such that
\[
 \Phi_\alpha(s)\ge c_0\,\operatorname{dist}(s,\Xi_\alpha)^2
\]
uniformly in the stated range of $\alpha$.  This proves \eqref{eq:certificate-coercive}.  \qedhere
\end{proof}

\begin{theorem}[Stable extremal quadrature]\label{thm:stable-quadrature}
Fix $d\ge2$ and $t\ge2$.  There exist constants $\varepsilon_0,C>0$, depending only on $d$ and $t$, such that the following holds.  Let $\nu$ be a probability measure supported on $[-1,\alpha]$, where
\[
 \eta_{t,d}\le\alpha\le\eta_{t,d}+\varepsilon_0,
 \qquad
 \delta:=\alpha-\eta_{t,d},
\]
and suppose that $\nu$ and $\mu_d$ have the same moments through degree $t$.  Then
\begin{equation}\label{eq:mean-square-clustering}
 \boxed{
 \int \operatorname{dist}(s,\Xi_{t,d})^2\,d\nu(s)
 \le C\delta,}
\end{equation}
and
\begin{equation}\label{eq:W1-stable}
 \boxed{
 W_1(\nu,\nu_{t,d}^*)\le C\sqrt\delta.}
\end{equation}

More precisely, choose the nearest-node projection from $[-1,\alpha]$ onto the moving set $\Xi_\alpha$ of \cref{lem:certificate-coercive}, and let $E_k(\alpha)$ be its Voronoi cells.  If
\[
 w_k:=\nu(E_k(\alpha)),
\]
with the top cell indexed so that it is mapped from $\alpha$ back to $\eta_{t,d}$, then
\begin{equation}\label{eq:weight-stability}
 \boxed{
 \sum_{k=0}^n|w_k-\lambda_k|\le C\sqrt\delta.}
\end{equation}
\end{theorem}

\begin{proof}
By \cref{lem:certificate-coercive},
\[
 c_0\int\operatorname{dist}(s,\Xi_\alpha)^2\,d\nu(s)
 \le\int\Phi_\alpha\,d\nu
 =\delta.
\]
Let $\pi_\alpha$ be a measurable nearest-node projection and set
\[
 \widetilde\nu:=(\pi_\alpha)_\#\nu.
\]
The coupling $s\mapsto\pi_\alpha(s)$ gives
\begin{equation}\label{eq:projection-W1}
 W_1(\nu,\widetilde\nu)
 \le\left(\int|s-\pi_\alpha(s)|^2\,d\nu(s)\right)^{1/2}
 \le C\sqrt\delta.
\end{equation}
Move the top atom of $\widetilde\nu$ from $\alpha$ to $\eta_{t,d}$ and leave the remaining atoms fixed.  Denote the resulting measure by
\[
 \widehat\nu=\sum_{k=0}^n w_k\delta_{\xi_k}.
\]
Then
\[
 W_1(\widetilde\nu,\widehat\nu)\le\delta,
\]
so
\begin{equation}\label{eq:nu-hat-W1}
 W_1(\nu,\widehat\nu)\le C\sqrt\delta.
\end{equation}

It remains to compare the weights $w_k$ with the quadrature weights.  Let
\[
 V=(\xi_k^\ell)_{0\le\ell,k\le n}
\]
be the Vandermonde matrix of the distinct quadrature nodes.  Since $V$ is invertible, it suffices to compare the first $n$ moments.  For $1\le\ell\le n$, the function $s\mapsto s^\ell$ is $\ell$-Lipschitz on $[-1,1]$.  Hence, using moment matching of $\nu$ and exactness of $\nu_{t,d}^*$,
\[
 \left|\sum_k(w_k-\lambda_k)\xi_k^\ell\right|
 =\left|\int s^\ell\,d(\widehat\nu-\nu)\right|
 \le \ell W_1(\widehat\nu,\nu)
 \le C\sqrt\delta.
\]
The zeroth moment difference is zero.  Applying $V^{-1}$ proves \eqref{eq:weight-stability}.  Since two probability measures on the same finite node set have Wasserstein distance bounded by their $\ell^1$ weight difference, \eqref{eq:W1-stable} follows from \eqref{eq:nu-hat-W1} and \eqref{eq:weight-stability}.  Finally, $\operatorname{dist}(s,\Xi_{t,d})\le \operatorname{dist}(s,\Xi_\alpha)+\delta$.  Hence, after squaring and integrating, the bound for $\Xi_\alpha$ gives \eqref{eq:mean-square-clustering} (shrinking $\varepsilon_0$ so that $\delta\le1$).  \qedhere
\end{proof}

We now return to self-polar Minkowski design polytopes.  In the node-transitive setting, write
\[
 b_{ij}:=\ip{Ux_i}{x_j},
 \qquad
 \alpha=\frac rR,
 \qquad
 \delta:=\alpha-\eta_{t,d}.
\]
Recall that $A_{ij}=\alpha-b_{ij}$ is the dimensionless scaled slack matrix from \cref{thm:slack}; indeed, if $S$ denotes the usual slack matrix, then $A=ch\,S$ in the node-transitive case.

\begin{theorem}[Near-quadrature implies near-incidence rigidity]\label{thm:near-incidence}
Fix $d\ge2$ and $t\ge2$.  There exist constants $\varepsilon_0,C,c>0$, depending only on $d$ and $t$, such that every node-transitive self-polar Minkowski polytope satisfying
\[
 0\le\delta:=\frac rR-\eta_{t,d}\le\varepsilon_0
\]
has the following properties for every facet row $i$.

\begin{enumerate}[label=\textup{(\roman*)}]
\item The empirical row measure
\[
 \nu_i:=\frac1N\sum_{j=1}^N\delta_{b_{ij}}
\]
satisfies
\begin{equation}\label{eq:row-W1}
 W_1(\nu_i,\nu_{t,d}^*)\le C\sqrt\delta,
\end{equation}
and
\begin{equation}\label{eq:row-cluster}
 \frac1N\sum_{j=1}^N\operatorname{dist}(b_{ij},\Xi_{t,d})^2\le C\delta.
\end{equation}

\item Let $E_k(\alpha)$ be the moving Voronoi cells from \cref{thm:stable-quadrature} and define the integer cluster multiplicities
\[
 m_{ik}:=\#\{j:b_{ij}\in E_k(\alpha)\}.
\]
Then
\begin{equation}\label{eq:cluster-multiplicity}
 \boxed{
 \left|\frac{m_{ik}}N-\lambda_k\right|\le C\sqrt\delta
 \qquad(0\le k\le n).}
\end{equation}

\item If $\lambda_+$ is the weight of the largest quadrature node, then
\begin{equation}\label{eq:near-incidence-count}
 \boxed{
 \left|
 \frac1N\#\{j:A_{ij}\le C\sqrt\delta\}-\lambda_+
 \right|\le C\sqrt\delta.}
\end{equation}
Thus the top quadrature weight controls not only exact incidences at equality but also the number of vertices lying within $O_{d,t}(\sqrt\delta)$ normalized slack of the facet.
\end{enumerate}

Finally, there is an arithmetic stability gap:
\begin{equation}\label{eq:arithmetic-gap}
 \boxed{
 \delta\ge
 \frac{c}{N^2}
 \max_{0\le k\le n}
 \operatorname{dist}(N\lambda_k,\mathbb Z)^2.}
\end{equation}
\end{theorem}

\begin{proof}
For fixed $i$, the vector $Ux_i$ is a unit vector and $X$ is a spherical $t$-design.  Hence the empirical measure $\nu_i$ has the same moments as $\mu_d$ through degree $t$.  By \cref{prop:alpha}, its support lies in $[-1,\alpha]$ and its maximum is $\alpha$.  Therefore \cref{thm:stable-quadrature} gives \eqref{eq:row-W1}, \eqref{eq:row-cluster}, and \eqref{eq:cluster-multiplicity}.

It remains to sharpen the top cluster into a near-incidence statement.  Let $E_+(\alpha)$ be the Voronoi cell of the moving top node $\alpha$.  Because the next quadrature node is separated from $\eta_{t,d}$ by a positive gap, after decreasing $\varepsilon_0$ if necessary there is a constant $c_1>0$ such that on $E_+(\alpha)$ the extremal certificate from \cref{lem:certificate-coercive} satisfies
\begin{equation}\label{eq:linear-top-coercivity}
 \Phi_\alpha(s)\ge c_1(\alpha-s).
\end{equation}
Indeed, the auxiliary factor $q_*^2$ in the odd case, and $(1+s)q_*^2$ in the even case, is bounded below on the top Voronoi cell.  Consequently, for every $\rho>0$ smaller than one quarter of the top-node separation,
\[
 \nu_i\bigl(E_+(\alpha)\cap\{\alpha-s>\rho\}\bigr)
 \le \frac{\delta}{c_1\rho}.
\]
Taking $\rho=C_0\sqrt\delta$ with $C_0$ fixed and $\delta$ sufficiently small, points with $\alpha-s\le\rho$ lie in $E_+(\alpha)$.  Combining the preceding estimate with \eqref{eq:cluster-multiplicity} for the top cluster yields \eqref{eq:near-incidence-count}, because $A_{ij}=\alpha-b_{ij}$.

For the arithmetic statement, each $m_{ik}$ is an integer and \eqref{eq:cluster-multiplicity} gives
\[
 \operatorname{dist}(N\lambda_k,\mathbb Z)
 \le |N\lambda_k-m_{ik}|
 \le CN\sqrt\delta.
\]
Squaring and maximizing over $k$ proves \eqref{eq:arithmetic-gap}.  \qedhere
\end{proof}

\begin{remark}[What is new in the stability statement]\label{rem:stability-novelty}
Equality cases of the Fazekas--Levenshtein bound, especially for odd strength, have a substantial theory; see \cite{Borodachov2024} and the references therein.  The point of \cref{thm:stable-quadrature,thm:near-incidence} is different.  They quantify how a small support-threshold defect forces the entire one-dimensional moment measure to approach the extremal quadrature rule, and then use self-polar Minkowski geometry to convert that analytic stability into a statement about near-zero entries of a polytope slack matrix.  The arithmetic gap \eqref{eq:arithmetic-gap} is the quantitative counterpart of the exact integrality condition in \eqref{eq:integrality}.
\end{remark}

\section{Rigidity consequences}
\label{sec:rigidity}

\subsection{The sharp quadratic case}

The case $t=2$ admits a particularly elementary rigidity proof that does not use the general quadrature machinery.

\begin{theorem}[Sharp twisted coherence for $2$-designs]\label{thm:t2}
Let $X=\{x_1,\ldots,x_N\}\subset\Sph^{d-1}$ be a spherical $2$-design and $U\in O(d)$.  Then
\begin{equation}\label{eq:t2coherence}
 \max_{i,j}\ip{Ux_i}{x_j}\ge\frac1d.
\end{equation}
If equality holds, then $N=d+1$, $X$ is a regular simplex, and there is a permutation $\pi$ such that $Ux_i=-x_{\pi(i)}$ for every $i$.
\end{theorem}

\begin{proof}
Fix $i$ and write $b_j=\ip{Ux_i}{x_j}$ and $\alpha=\max_j b_j$.  The design identities give
\[
 \frac1N\sum_jb_j=0,
 \qquad
 \frac1N\sum_jb_j^2=\frac1d.
\]
Because $-1\le b_j\le\alpha$,
\[
 (b_j+1)(\alpha-b_j)\ge0,
\]
so
\[
 b_j^2\le(\alpha-1)b_j+\alpha.
\]
Averaging yields $1/d\le\alpha$.  Thus \eqref{eq:t2coherence} holds.

If equality holds globally, equality holds in every row inequality, hence every $b_j$ belongs to $\{-1,1/d\}$.  If a row contains $k$ occurrences of $-1$, the mean-zero equation gives
\[
 -k+\frac{N-k}{d}=0,
 \qquad\text{so}\qquad N=(d+1)k.
\]
For a fixed unit vector $Ux_i$, there can be at most one design point with inner product $-1$, because that point must equal $-Ux_i$.  Thus $k=1$ and $N=d+1$.  A spherical $2$-design with $d+1$ points has Gram matrix
\[
 G=\frac{d+1}{d}\left(\Id-\frac1{d+1}\mathbf1\mathbf1^T\right),
\]
so all off-diagonal inner products are $-1/d$; hence $X$ is a regular simplex.  \qedhere
\end{proof}

\begin{corollary}[Quadratic self-polar equality]\label{cor:t2selfpolar}
In the setting of \cref{thm:quadcomb} with $t=2$,
\[
 \frac rR\ge\frac1d.
\]
Equality holds only for the regular simplex, in which case each facet has $d$ vertices and the Gauss--Radau weights are
\[
 \frac1{d+1}\quad\text{at }-1,
 \qquad
 \frac d{d+1}\quad\text{at }\frac1d.
\]
\end{corollary}

\subsection{Simple self-polarity}

\begin{theorem}[Simple self-polar rigidity]\label{thm:simple}
Let $d\ge3$ and let $X\subset\Sph^{d-1}$ be a spherical $t$-design, $t\ge2$.  If its Minkowski polytope $P_X$ is simple and self-polar up to scale, then $P_X$ is a $d$-simplex, $N=d+1$, and $X$ is a regular simplex.  In particular, the design strength is at most two.
\end{theorem}

\begin{proof}
The polar of a simple polytope is simplicial.  Self-polarity makes $P_X$ combinatorially equivalent to $P_X^\circ$, so $P_X$ is also simplicial.  A $d$-polytope with $d\ge3$ that is both simple and simplicial is a simplex; see Ziegler \cite{Ziegler1995}.  Hence $N=d+1$.  The first two design moments then give the regular-simplex Gram matrix exactly as in the proof of \cref{thm:t2}.

A regular simplex is not a spherical $3$-design.  Indeed, choosing $a=x_1$,
\[
 \frac1{d+1}\sum_{i=1}^{d+1}\ip{x_i}{a}^3
 =\frac1{d+1}\left(1-d\frac1{d^3}\right)
 =\frac{1-d^{-2}}{d+1}>0,
\]
whereas the spherical integral of the odd cubic is zero.  \qedhere
\end{proof}

\begin{corollary}
If $t\ge3$, every self-polar Minkowski polytope of a spherical $t$-design in dimension $d\ge3$ is necessarily non-simple.  Thus high-strength self-polarity forces nongeneric facet--vertex incidences.
\end{corollary}

\subsection{A complete node-transitive obstruction in dimension three}

\begin{theorem}[Three-dimensional node-transitive rigidity]\label{thm:3d}
Let $X\subset\Sph^2$ be a node-transitive spherical $t$-design, $t\ge2$, and let $P_X\subset\R^3$ be its Steiner-normalized Minkowski polytope.  If $P_X$ is self-polar up to scale, then $P_X$ is a tetrahedron and $X$ is a regular tetrahedral $2$-design.  Consequently,
\begin{equation}\label{eq:no3d}
 \boxed{t\ge3\quad\Longrightarrow\quad P_X\text{ is not self-polar}.}
\end{equation}
\end{theorem}

\begin{proof}
Self-polarity gives
\[
 f_0(P_X)=f_2(P_X)=N.
\]
By \cref{prop:constant-h,thm:bidual}, all vertices have the form
$RU^Tx_j$ with the common circumradius $R$.  If $G$ is transitive on $X$, the conjugate group $U^TGU$ acts transitively on the vertex set, so $P_X$ is vertex-transitive.  Therefore every vertex has the same graph degree, say $q$.

Euler's relation and $f_0=f_2=N$ give
\[
 N-f_1+N=2,
 \qquad\text{hence}\qquad f_1=2N-2.
\]
Counting vertex--edge incidences,
\[
 qN=2f_1=4N-4,
\]
so
\[
 q=4-\frac4N.
\]
Since $q$ is an integer and $N\ge4$, necessarily $N=4$ and $q=3$.  Hence $P_X$ is a tetrahedron.  A spherical $2$-design with four points in $\Sph^2$ is a regular tetrahedron, and \cref{thm:simple} shows that its strength cannot exceed two.  \qedhere
\end{proof}

\section{Regular examples: the simplex and the $24$-cell}
\label{sec:regular}

The preceding results become especially transparent for vertex sets of regular convex polytopes.  We use the classical classification; see Coxeter \cite{Coxeter1973}.

\begin{theorem}[Regular self-polar classification in the design setting]\label{thm:regularclass}
Let $d\ge3$ and let $X$ be the vertex set of a regular convex $d$-polytope $C$ which is also a spherical design of strength at least two.  Let $P_X$ be its Minkowski polytope.  Then $P_X$ is self-polar up to scale if and only if $C$ is self-dual up to similarity.  Consequently, among regular convex polytopal designs in dimensions $d\ge3$, the self-polar Minkowski examples are precisely:
\begin{enumerate}[label=\textup{(\roman*)}]
\item the regular simplex, in every dimension;
\item the regular $24$-cell, in dimension four.
\end{enumerate}
\end{theorem}

\begin{proof}
The regular symmetry group is vertex-transitive, so \cref{prop:constant-h} gives $P_X=hC^\circ$.  Hence self-polarity of $P_X$ is equivalent to self-duality of $C$ up to orthogonal similarity.  The classification of regular convex polytopes says that the simplex is self-dual in every dimension, the cube and cross-polytope form a dual pair, and in dimension four the $24$-cell is the unique additional self-dual regular polytope; the $120$-cell and $600$-cell form a dual pair.  \qedhere
\end{proof}

\subsection{The simplex}

Let $X=\{x_1,\ldots,x_{d+1}\}$ be a regular simplex on $\Sph^{d-1}$, so
\[
 \ip{x_i}{x_j}=-\frac1d\quad(i\ne j).
\]
Then
\[
 C^\circ=-dC.
\]
Thus for $P=hC^\circ$ one has
\[
 P^\circ=-\frac1{d h^2}P.
\]
The self-polar map may be chosen $U=-\Id$, and
\[
 \alpha=ch^2=\frac1d=\eta_{2,d}.
\]
This is exactly the equality case of \cref{cor:t2selfpolar}.

\subsection{The exceptional degree-five equality of the $24$-cell}

Consider the standard unit-sphere realization
\begin{equation}\label{eq:24set}
 X_{24}
 =\{\pm e_1,\ldots,\pm e_4\}
 \cup
 \left\{\frac12(\varepsilon_1,\varepsilon_2,\varepsilon_3,\varepsilon_4):
 \varepsilon_i\in\{\pm1\}\right\}
 \subset\Sph^3.
\end{equation}
This is the vertex set of a regular $24$-cell and is a spherical $5$-design; see, for example, Xiang \cite{Xiang2022}.  The $24$-cell is self-dual.  Its attainment of the relevant Fazekas--Levenshtein covering threshold is also visible in the spherical-polarization analysis of \cite{BoyvalenkovEtAl2023}; below we reinterpret the equality multiplicities as facet--vertex incidence numbers.

Define the orthogonal involution
\begin{equation}\label{eq:U24}
 U=\frac1{\sqrt2}
 \begin{pmatrix}
 1&1&0&0\\
 1&-1&0&0\\
 0&0&1&1\\
 0&0&1&-1
 \end{pmatrix}.
\end{equation}
Then $UX_{24}$ is the $D_4$ root realization
\[
 \left\{\frac1{\sqrt2}(\pm e_i\pm e_j):1\le i<j\le4\right\},
\]
which is the dual $24$-cell realization.  In fact, if $C=\conv X_{24}$, then
\begin{equation}\label{eq:24dual}
 C^\circ=\sqrt2\,UC.
\end{equation}

For $x=e_1$, one has
\[
 Ue_1=\frac1{\sqrt2}(1,1,0,0).
\]
A direct count in \eqref{eq:24set} gives
\begin{equation}\label{eq:24distribution}
 \#\left\{y\in X_{24}:\ip{Ue_1}{y}=\frac1{\sqrt2}\right\}=6,
\end{equation}
\[
 \#\{y:\ip{Ue_1}{y}=0\}=12,
 \qquad
 \#\left\{y:\ip{Ue_1}{y}=-\frac1{\sqrt2}\right\}=6.
\]
By regular symmetry the same pattern occurs in every corresponding row.

Since $d=4$ and $t=5$,
\[
 \eta_{5,4}=\sqrt{\frac3{6}}=\frac1{\sqrt2},
\]
and \eqref{eq:t5weights} becomes
\[
 (\lambda_-,\lambda_0,\lambda_+)=(1/4,1/2,1/4).
\]
Thus \eqref{eq:24distribution} is exactly the Gaussian equality distribution predicted by \cref{thm:quadcomb}.  The top multiplicity is
\[
 N\lambda_+=24\cdot\frac14=6,
\]
which is precisely the number of vertices of an octahedral facet of the $24$-cell.

\begin{proposition}[The $24$-cell saturates the self-polar design bound]\label{prop:24}
For the Minkowski polytope associated with $X_{24}$,
\[
 \frac rR=\eta_{5,4}=\frac1{\sqrt2}.
\]
Every facet row realizes the three-node Gaussian rule with weights $(1/4,1/2,1/4)$, and every facet has six vertices.
\end{proposition}

\begin{proof}
By \eqref{eq:24dual} and \cref{prop:constant-h}, the self-polar incidence level is $1/\sqrt2$.  The explicit distribution above completes the proof.  \qedhere
\end{proof}

\begin{remark}
The $24$-cell illustrates why \cref{thm:simple} is not a nonexistence theorem for high-strength self-polarity.  The $24$-cell is neither simple nor simplicial; it occupies the exceptional combinatorial regime in which strength five and self-duality coexist.  The Gaussian quadrature pattern explains quantitatively how its octahedral facets fit the design moments.
\end{remark}

\section{Further directions}
\label{sec:future}

The results above suggest a program in which polar duality is treated as a structural operation on spherical designs rather than merely on their underlying polytopes.  We record several problems that appear particularly natural.

\begin{problem}[Classification of self-polar Minkowski design polytopes]
Classify spherical $t$-designs $X\subset\Sph^{d-1}$ for which the Steiner-normalized Minkowski polytope satisfies $P_X^\circ=cUP_X$.  The three-dimensional node-transitive case is settled by \cref{thm:3d}, while \cref{thm:regularclass} settles regular convex polytopal designs.  The first genuinely open regime is therefore nonregular, non-simple self-polarity in dimension $d\ge4$.
\end{problem}

\begin{problem}[Equality in the quadrature-combinatorics bound]
Classify node-transitive self-polar design polytopes satisfying
\[
 r/R=\eta_{t,d}.
\]
By \cref{thm:quadcomb}, every facet row then realizes a Gaussian or Gauss--Radau quadrature rule exactly.  The integrality constraints $N\lambda_k\in\mathbb Z$ provide immediate arithmetic obstructions.  Determine whether the regular simplex and the $24$-cell are isolated equality models in natural classes of designs.
\end{problem}

\begin{problem}[Uniform-in-strength stability and classification]
\Cref{thm:stable-quadrature,thm:near-incidence} give local stability for each fixed pair $(d,t)$, with constants depending on the conditioning and separation of the relevant quadrature nodes.  Determine sharp dependence on $t$ for fixed dimension $d$, and use it to classify sequences of self-polar Minkowski design polytopes satisfying
\[
 \frac rR-\eta_{t,d}\longrightarrow0
 \qquad (t\to\infty).
\]
In particular, determine whether the near-incidence clusters forced by \cref{thm:near-incidence} can be upgraded, under natural combinatorial hypotheses, to exact incidence rigidity for all sufficiently large $t$.
\end{problem}

\begin{problem}[Nontransitive self-polarity]
In the absence of node transitivity the support numbers $h_i$ are nonconstant and the incidence threshold becomes the rank-one matrix $c h_i h_j$.  The slack identity
\[
 A=c\,hh^T-X^TU^TX\ge0
\]
suggests a weighted one-dimensional moment problem coupled to the geometry of $h$.  Develop sharp inequalities and classification results directly from this structured nonnegative factorization.
\end{problem}

\begin{problem}[Improved Minkowski stability and polar moment windows]
The current universal three-dimensional rate $\dH(P_t,B)=O(t^{-1/2})$ gives the polar isotropy window $m=o(\sqrt t)$.  For special well-conditioned designs, determine whether the wavelength-scale estimate $O(t^{-1})$ holds.  By \cref{thm:momenttransfer}, such an improvement would automatically enlarge the polar moment window to $m=o(t)$.
\end{problem}

The emerging picture may be summarized as
\[
\boxed{
\begin{aligned}
&\text{spherical design}\longrightarrow\text{Minkowski facet geometry}
\longrightarrow\text{polar vertex geometry}\\
&\hspace{2.5cm}\longrightarrow\text{self-polar slack structure}
\longrightarrow\text{quadrature-controlled combinatorics}.
\end{aligned}}
\]
The first arrow is governed by the Minkowski problem, the second by polarity, and the last two by the interaction between spherical moment exactness and polytope duality.  The exceptional behavior of the $24$-cell indicates that this interaction is sensitive not only to harmonic strength but also to the fine incidence geometry of the polytope.

\section*{Acknowledgements}
The author thanks Yeyao Hu for encouragement and helpful discussions.


\begin{thebibliography}{99}

\bibitem{An2026}
C.~An,
\newblock \emph{Minkowski polytopes of spherical designs: High-order isotropy and quantitative sphericity},
\newblock arXiv:2608.11570, 2026.

\bibitem{Borodachov2024}
S.~Borodachov,
\newblock Odd strength spherical designs attaining the Fazekas--Levenshtein bound for covering and universal minima of potentials,
\newblock \emph{Aequationes Math.} \textbf{98} (2024), 509--533.
\newblock doi:10.1007/s00010-024-01036-6.

\bibitem{BoyvalenkovEtAl2023}
P.~G. Boyvalenkov, P.~D. Dragnev, D.~P. Hardin, E.~B. Saff, and M.~M. Stoyanova,
\newblock On polarization of spherical codes and designs,
\newblock \emph{J. Math. Anal. Appl.} \textbf{524} (2023), no.~1, Article 127065.
\newblock doi:10.1016/j.jmaa.2023.127065.

\bibitem{Coxeter1973}
H.~S.~M. Coxeter,
\newblock \emph{Regular Polytopes}, 3rd ed.,
\newblock Dover Publications, New York, 1973.

\bibitem{DGS1977}
P.~Delsarte, J.-M. Goethals, and J.~J. Seidel,
\newblock Spherical codes and designs,
\newblock \emph{Geom. Dedicata} \textbf{6} (1977), 363--388.

\bibitem{Ehler2025}
M.~Ehler,
\newblock \emph{Hybrid spherical designs},
\newblock arXiv:2502.07720, 2025.

\bibitem{FazekasLevenshtein1995}
G.~Fazekas and V.~I. Levenshtein,
\newblock On upper bounds for code distance and covering radius of designs in polynomial metric spaces,
\newblock \emph{J. Combin. Theory Ser. A} \textbf{70} (1995), no.~2, 267--288.
\newblock doi:10.1016/0097-3165(95)90093-4.

\bibitem{GouveiaEtAl2013}
J.~Gouveia, R.~Grappe, V.~Kaibel, K.~Pashkovich, R.~Z.~Robinson, and R.~R.~Thomas,
\newblock Which nonnegative matrices are slack matrices?,
\newblock \emph{Linear Algebra Appl.} \textbf{439} (2013), no.~10, 2921--2933.
\newblock doi:10.1016/j.laa.2013.08.009.

\bibitem{GouveiaLourenco2023}
J.~Gouveia and B.~F. Louren\c{c}o,
\newblock Self-dual polyhedral cones and their slack matrices,
\newblock \emph{SIAM J. Matrix Anal. Appl.} \textbf{44} (2023), no.~3, 1096--1121.
\newblock doi:10.1137/22M1519869.

\bibitem{Jensen2021}
A.~Jensen,
\newblock Self-polar polytopes,
\newblock in \emph{Polytopes and Discrete Geometry}, Contemp. Math. \textbf{764},
\newblock Amer. Math. Soc., Providence, RI, 2021, pp.~101--124.
\newblock doi:10.1090/conm/764/15333.

\bibitem{Schneider2014}
R.~Schneider,
\newblock \emph{Convex Bodies: The Brunn--Minkowski Theory}, 2nd expanded ed.,
\newblock Encyclopedia of Mathematics and its Applications 151, Cambridge University Press, Cambridge, 2014.

\bibitem{Szego1975}
G.~Szeg\H{o},
\newblock \emph{Orthogonal Polynomials}, 4th ed.,
\newblock American Mathematical Society Colloquium Publications, Vol.~23,
\newblock American Mathematical Society, Providence, RI, 1975.

\bibitem{Waldron2025}
S.~Waldron,
\newblock \emph{Real and complex spherical designs and their Gramian},
\newblock arXiv:2511.07452, 2025.

\bibitem{Xiang2022}
Z.~Xiang,
\newblock Explicit spherical designs,
\newblock \emph{Algebraic Combinatorics} \textbf{5} (2022), no.~2, 347--369.
\newblock doi:10.5802/alco.213.

\bibitem{Ziegler1995}
G.~M. Ziegler,
\newblock \emph{Lectures on Polytopes},
\newblock Graduate Texts in Mathematics 152, Springer-Verlag, New York, 1995.

\end{thebibliography}
\end{document}